\documentclass{amsart}

\usepackage{amssymb,amsmath,amsfonts,amsthm,epsfig,amscd}
\usepackage{commath}
\usepackage{stmaryrd}
\usepackage[all,cmtip,poly]{xy}
\usepackage{svn}
\usepackage{mathtools}
\usepackage{mathrsfs}
\usepackage{tikz-cd}
\usepackage{hyperref}
\usepackage{nameref}
\usepackage{dsfont}
\usepackage{stackrel}

\newtheorem{lem}{Lemma}[section]

\newtheorem{definition}[lem]{Definition}

\newtheorem{cor}[lem]{Corollary}

\newtheorem{theorem}[lem]{Theorem}
\newtheorem{prop}[lem]{Proposition}

\newtheorem{conj}[lem]{Conjecture}

\theoremstyle{remark}
\newtheorem{rem}[lem]{Remark}

\newcommand{\GL}{\mathrm{GL}}
\newcommand{\SL}{\mathrm{SL}}

\def\UgD{U(\frak{g}_D)}

\def\wUgR{\widehat{U(\frak{g}_R)}}
\def\wUgK{\widehat{U(\frak{g})}_K}
\def\wUgnK{\widehat{U(\frak{g})_{n,K}}}

\def\wtM{\widetilde{M}}
\def\whM{\widehat{M}}

\def\gr{\mathrm{gr}}
\def\Gr{\mathrm{Gr}}
\def\Hom{\mathrm{Hom}}

\def\for{\mathrm{for}}

\newcommand{\la}{\mathrm{la}}
\newcommand{\an}{\mathrm{an}}

\newcommand{\CC}{{\mathbb C}}

\newcommand{\FF}{{\mathbb F}}

\newcommand{\NN}{{\mathbb N}}

\newcommand{\QQ}{{\mathbb Q}}
\newcommand{\RR}{{\mathbb R}}

\newcommand{\ZZ}{{\mathbb Z}}

\newcommand{\cO}{{\mathcal O}}

\makeatletter
\tikzset{
  column sep/.code=\def\pgfmatrixcolumnsep{\pgf@matrix@xscale*(#1)},
  row sep/.code   =\def\pgfmatrixrowsep{\pgf@matrix@yscale*(#1)},
  matrix xscale/.code=%
    \pgfmathsetmacro\pgf@matrix@xscale{\pgf@matrix@xscale*(#1)},
  matrix yscale/.code=%
    \pgfmathsetmacro\pgf@matrix@yscale{\pgf@matrix@yscale*(#1)},
  matrix scale/.style={/tikz/matrix xscale={#1},/tikz/matrix yscale={#1}}}
\def\pgf@matrix@xscale{1}
\def\pgf@matrix@yscale{1}
\makeatother

\title{Simple Iwasawa modules with large canonical dimension of quaternion algebra over $\QQ_p$}
\begin{document}
\author{Weibo Fu}
\email{wfu@math.princeton.edu}
\date{April,~2025}

\begin{abstract}
We construct certain absolutely irreducible Banach representations of the quaternion algebra of large canonical (Gelfand-Kirillov) dimension which yield counterexamples to a natural conjecture of Dospinescu-Schraen on the existence of an infinitesimal character and topological finite length for the locally analytic vectors of an absolutely irreducible Banach representation.
\end{abstract}
\maketitle
\tableofcontents

\section{Introduction}\label{intro}
Let $G$ be a connected reductive group over a local field and let $\frak{g}$ be its Lie algebra. 
The center $Z(\frak{g})$ of the enveloping algebra of $U(\frak{g})$ is of significant importance in the representation theory of $G$.

When the local field is $\RR$, Harish-Chandra established an equivalence between irreducible unitary representations of $G$ and irreducible infinitesimally unitary $(\frak{g}, K)$ modules by passing to $K$ finite vectors. It is clear from Schur's lemma that there exists a natural map
\[ \widehat{G(\RR)} \to \Hom_{\CC\mathrm{-alg}} (Z(\frak{g}),\CC), \] where $\widehat{G(\RR)}$ is the unitary dual of $\widehat{G(\RR)}$.

Understanding the relationship between “Banach” and “analytic” aspects has been beneficial in the development of $p$-adic Langlands program for $\GL_2(\QQ_p)$. Recently, Emerton, Gee, and Hellmann \cite{EGH23} formulate some conjectures separately in the “Banach” and “analytic” cases and their conjectural compatiblity in the emerging “categorical” $p$-adic Langlands program.

However, the concepts of “Banach” and “analytic” in the $p$-adic case for representation theory of $p$-adic analytic groups and their relationships are much more unclear and mysterious than in the classical Archimedean case. 
Let $L$ be a sufficiently large extension of $\QQ_p$ such that $G$ is split over. Let $\widehat{G(\QQ_p)}_L$ be the set of isomorphism classes of absolutely irreducible $L$-Banach space representations and $\widehat{G(\QQ_p)}_L^{\mathrm{unit}}$ be the set of isomorphism classes of absolutely irreducible unitary $L$-Banach space representations. For both sets $\widehat{G(\QQ_p)}_L$ and $\widehat{G(\QQ_p)}_L^{\mathrm{unit}}$, there exist irreducible Banach space representations whose locally analytic vectors are reducible (\cite{ST01}, \cite{Col10}, \cite{LXZ12}).
Even though irreducibility is not preserved, Dospinescu-Schraen \cite{DS13} asks the following natural and intriguing major question on the existence of an infinitesimal character on locally analytic vectors:
\begin{conj}\label{infinitesimal char on locally analytic vectors}
Let $\Pi$ be an absolutely irreducible admissible $L$-Banach space representation of $G(\QQ_p)$. Then the space $\Pi^\la$ of locally analytic vectors in $\Pi$ is a locally analytic $G(\QQ_p)$-representation of finite length, having an infinitesimal character.
\end{conj}

A literal consequence of their conjecture is the existence of a natural map
\[ \widehat{G(\QQ_p)}_L \to \Hom_{L\mathrm{-alg}} (Z(\frak{g})_L, L). \]
Not much was known about this conjecture beyond the torus case.
Dospinescu \cite{Dos12} shows the existence of
\[ \widehat{\GL_2(\QQ_p)}_L^{\mathrm{unit}} \to \Hom_{L\mathrm{-alg}} (Z(\frak{g})_L, L) \] by fully using the $p$-adic local Langlands correspondence for $\GL_2(\QQ_p)$.

But in this paper, we exhibit counterexamples of this conjecture for the quaternion algebra $D$. For simplicity, we will use $D^\times$ for the $\QQ_p$ points of the reductive group $D^\times$.
\begin{theorem}\label{counterexamples}
There exists an absolutely irreducible admissible representation $V$ of $D^\times$ such that its locally analytic vectors $V^\la$ have infinite length without an infinitesimal character.
\end{theorem}

As our technique crucially uses the anisotropicness of $D^\times$ modulo its center, we may still ask for a finite extension $F$ over $\QQ_p$ and a quasi-split group $G$ over $F$:
\begin{conj}
Does Conjecture \ref{infinitesimal char on locally analytic vectors} hold for $G(F)$?
\end{conj}

Ardakov-Wadsley, in their breakthrough work \cite{AW13}, interpreted the canonical dimension of the Banach space representations geometrically and proved the first non-trivial lower bounds of admissible Banach space representations by half of the smallest possible dimension of coadjoint orbits for a split group over $\QQ_p$.

Our constructions for the representations in Theorem \ref{counterexamples} have canonical dimension equal to 2, and are first examples of absolutely irreducible Banach space representations of $G(\QQ_p)$ whose dimension is larger than half of the largest dimension of coadjoint orbits of $G_\CC$, to the knowledge of the author.

Although there is no "holonomic" concept for admissible Banach space representations, our constructed admissible irreducible Banach space representations are pathological as Dospinescu-Pa\v{s}k\={u}nas-Schraen proved that representations from the $p$-adic Jacquet-Langlands correspondence have Gelfand-Kirillov dimension 1 \cite[Thm 1.1]{DPS22}. Hu-Wang proved both the Gelfand-Kirillov dimension 1 and the finite length properties of the representations obtained from the $p$-adic Jacquet-Langlands correspondence under certain assumptions in \cite{HW24a}, \cite{HW24b}.

We briefly describe our simple construction of the representation for a compact subgroup $D_1 := 1 + p \cO_D \subset D^\times$, where $\cO_D$ is the ring of integers of quaternion algebra over $\QQ_p$. 
Let $\overline{D_1}$ be $D_1$ modulo its center. It is a compact 3-dimensional $p$-adic analytic group, isomorphic to $\ZZ_p^3$. 
Pick a rank 1 subgroup $L \subset \overline{D_1}$ such that
\[ L \cong \ZZ_p, ~ \overline{D_1} / L \cong \ZZ_p^2, \] with a generator $l \in L$ and a generic character $\lambda: L \to K$ where $K$ is an extension of $\QQ_p$.
Most efforts of this paper are to prove
\[ V(L, \lambda) := \{ f \in C(\overline{D_1}, K) ~ | ~ f(gt) = \lambda(t)f(g) ~~ \for ~ \forall ~ t \in L,  ~ g \in \overline{D_1} \} \]
is an irreducible admissible Banach space representation of $\overline{D_1}$. Or equivalently,
\[ M(L,\lambda): K[[\overline{D_1}]] / K[[\overline{D_1}]] \cdot (l-\lambda(l)) \] is a simple left Iwasawa module over the Iwasawa algebra $K[[\overline{D_1}]]$.

It was also widely believed that any irreducible $D$-module is automatically holonomic until J.T. Stafford constructed a counterexample in 1983 (See the introduction of \cite{BL88}).
Our result is similar in spirit to Stafford's construction of non-holonomic simple modules for Weyl algebras with large Gelfand-Kirillov dimensions \cite{Sta85}. Bernstein-Lunts later proved the genericity of principally generated maximal left ideal in the 2 dimensional case. In contrast, we consider principally generated maximal left ideal of the Iwasawa algebra in our result, and put a genericity condition on $\lambda$.

We also highlight that simpleness of the Iwasawa module $M(L,\lambda)$ and its dual statement in Banach space representation is purely a question in Iwasawa algebra or $p$-adic Banach representation theory of a $p$-adic analytic group. It is necessary for us to pass our problem to the locally analytic side and attack it there. This essential passage to the analytic side has come to fruition in addressing questions in Banach space representations, e.g. \cite{AW13}, \cite{Fu24}. In this paper, we not only need to consider analytic vectors of fixed radius but also locally analytic vectors of any possibly small radius.

We now further outline the proof strategy. Consider the locally analytic vector subspace $V(L, \lambda)^\la$, which is dense in $V(L, \lambda)$ by \cite[Thm 7.1]{ST03}.
We prove that for any infinitesimal character $\chi$, the subrepresentation $V(L, \lambda)^\la[\chi] \subset V(L, \lambda)^\la$ cut out by $\chi$ is nonzero, with a desired direct product structure, which we will state in Prop \ref{M decomposition}.
This structure implies that for any nonzero function $f \in V(L, \lambda)^\la[\chi] \subset C^\la(\overline{D_1}, K)$ and an open closed subset $U \subset \overline{D_1} / L$, the product $f \cdot \chi_{U \cdot L}$ of $f$ and characteristic function $\chi_{U \cdot L}$ supported in $U \cdot L$, belongs to $V(L, \lambda)^\la[\chi]$.
On the other hand, it is clear that such vectors $\{f \cdot \chi_{U \cdot L}\}$ are dense in $V(L, \lambda)$. We can therefore deduce the density of $V(L, \lambda)^\la[\chi]$ in $V(L, \lambda)$ for any $\chi$.
A key ingredient for our structural result Prop \ref{M decomposition} is a concrete computational estimate for the growth of $p$-adic valuation of coefficients of certain group elements expressed in the completed universal enveloping algebra (Prop \ref{g not in id}).

We also want to point out that a similar statement that $C^\an(\overline{D_1}, K)[\chi]$ is dense in $C(\overline{D_1}, K)$ can be viewed as a dual statement of the main results in \cite{AW14}, which plays an important role in some recent works such as \cite{DPS22}, \cite{Fu24} for different local questions with global applications. Pan proved some of these global applications addressed in \cite{DPS20}, \cite{DPS22} with geometrical methods in \cite{Pan22}, \cite{Pan25}.

\subsection*{Acknowledgments}
I would like to thank Gabriel Dospinescu, Yongquan Hu, and Benjamin Schraen and for helpful discussions and feedbacks. 

\section{Notation}
For a pro-$p$ group $G$, let \[ G^p := \left\langle g^p \mid g \in G\right\rangle. \]
We refer to \cite[Def 4.1]{DDSMS99} for the definition and basic properties of pro-$p$ uniform groups.

Let $\QQ_p \subseteq K$ be a finite extension $K$ over $\QQ_p$.
For a locally $\QQ_p$-analytic group $G$, let $C(G, K)$ and $C^\la(G, K)$, $C^{\text{sm}}(G, K)$ be respectively the space of $K$-valued continuous functions on $G$, the space of locally analytic functions on $G$, and the space of smooth functions on $G$, all valued in $K$.

Let $G$ be a uniform pro-$p$ group. 
We define the completed group rings by \[ \ZZ_p[[G]] := \lim _{\longleftarrow} \ZZ_p[G / N], ~ K[[G]] := K \otimes_{\ZZ_p} \ZZ_p[[G]], \]
where $N$ runs over all the open normal subgroups $N$ of $G$.

Lazard \cite{Laz65} defines a $\ZZ_p$-Lie algebra $L_G$ associated to $G$ (see also \cite[\S 4.5]{DDSMS99}).
We refer to \cite[\S 4]{Fu24} for a brief recall of some basic facts about $L_G$ with \begin{eqnarray}\label{zp Lie} \frak{g}_{\ZZ_p} := \frac{1}{p} L_G, \end{eqnarray} as well as the distribution algebra $D(G, K)$ with various completions $D_r(G, K)$ there.

Given a minimal topological generating set $\{g_1,\cdots,g_d\}$ of $G$, each element of $G$ can be written uniquely in the form $g_1^{\lambda_1}\cdots g_d^{\lambda_d}$ for some $\lambda_1,\cdots,\lambda_d \in \ZZ_p$.
Let $b_i := g_i - 1 \in R[G]$, and write \begin{eqnarray}\label{balpha} \mathbf{b}^{\alpha}=b_{1}^{\alpha_{1}} \cdots b_{d}^{\alpha_{d}} \in R[G] \end{eqnarray} for any $d$-tuple $\alpha \in \NN^d$.
We write $|\alpha| := \sum_{i=1}^d \alpha_i$.

The distribution algebra $D(G, K)$ can be identified with power series in $b_1,\cdots, b_d$ with convergence conditions
\[ D(G, K) = \left\{\sum_{\alpha \in \mathbb{N}^{d}} \lambda_{\alpha} \mathbf{b}^{\alpha} \mid \lambda_\alpha \in K, ~ \mathrm{and} ~ \for ~ \forall ~ 0 < r < 1, ~ \sup_{\alpha \in \mathbb{N}^{d}} |\lambda_\alpha| r^{|\alpha|} < \infty \right\}. \]

The Fr\'echet topology of $D(G, K)$ is defined by the family of norms
\begin{eqnarray}\label{r norm} \|\lambda\|_{r}:=\sup _{\alpha \in \mathbb{N}^{d}}\left|\lambda_{\alpha}\right| r^{|\alpha|} \end{eqnarray} for $0 < r < 1$, where the absolute value $|\cdot|$ is normalized as usual by $|p| = p^{-1}$. We let
\[ D_r(G, K) := \text{completion of $D(G, K)$ with respect to the norm} ~ \| ~ \|_r. \]
As a $K$-Banach space,
\begin{eqnarray}\label{Dr series}
 D_r(G, K) = \left\{\sum_{\alpha \in \mathbb{N}^{d}} \lambda_{\alpha} \mathbf{b}^{\alpha} \mid \lambda_\alpha \in K, ~ \sup_{\alpha \in \mathbb{N}^{d}} |\lambda_\alpha| r^{|\alpha|} < \infty \right\}. 
\end{eqnarray}

We consider natural extensions of the Iwasawa algebra $K[[G]]$ as 
\begin{eqnarray}\label{ST dis} K[[G]] \to D(G, K) \to D_r(G, K), \end{eqnarray} by \cite[Thm 4.11]{ST03}, and microlocalisation map
\begin{eqnarray}\label{micro} K[[G]] \to \wUgK \end{eqnarray} by construction of \cite[\S 10]{AW13}, where the completed enveloping algebras are defined as
\[ \widehat{U(\frak{g}_{\ZZ_p})} :=\varprojlim_{a} \left(\frac{U(\mathfrak{g}_{\ZZ_p})}{p^{a} U(\mathfrak{g}_{\ZZ_p})}\right), ~ \wUgK := \widehat{U(\frak{g}_{\ZZ_p})}\otimes_{\ZZ_p} K. \]
For a finitely generated $K[[G]]$-module $\wtM$, we use $\whM := \wUgK \otimes_{K[[G]]} \wtM$ to denote the \emph{microlocalisation} of $\wtM$.

\section{A uniform pro-$p$ subgroup of $p$-adic Lie group $D^\times$ and its associated Lie algebra}\label{padic quaternion}
Let $p \geq 5$ be a prime. 
Let $D$ be the unique non-split central division (quaternion) algebra of dimension 4 over $\QQ_p$.
For the ramified quadratic extension $\QQ_p(\sqrt{p})$, there is an embedding $\QQ_p(\sqrt{p}) \hookrightarrow D$ with elements $\tau \in D^\times$, $\iota \in \ZZ_p^\times \backslash (\ZZ_p^\times)^2$ such that $D = \QQ_p(\sqrt{p}) \oplus \QQ_p(\sqrt{p}) \tau$, $\tau^2 = \iota$, and conjugation of $\tau$ stabilizes $\QQ_p(\sqrt{p})$ and induces the nontrivial Galois automorphism \[ \tau x \tau^{-1} = \bar{x}, ~~ \forall x \in \QQ_p(\sqrt{p}). \]
We have the following embedding of $\QQ_p$-algebras \begin{eqnarray}\label{quater embed}
D \xrightarrow{\imath} M_2(\QQ_p(\sqrt{p})), \quad a+b \tau \mapsto \begin{pmatrix} a & \iota b \\ \bar{b} & \bar{a}\end{pmatrix},
\end{eqnarray} inducing the isomorphism \[ D \otimes_{\QQ_p} \QQ_p(\sqrt{p}) \simeq M_2(\QQ_p(\sqrt{p})). \]

We may define $v_D(a) := v_p(\det(\imath(a)))$, where $v_p$ is the $p$-adic valuation on $\QQ_p(\sqrt{p})$ normalized so that $v_p(p) = 1$.
Let $I$ be the identity of $D \xrightarrow{\imath} M_2(\QQ_p(\sqrt{p}))$.
Let $\cO_D := \{ a \in D ~ | ~ v_D(a) \geq 0 \}$ be the ring of integers with the maximal ideal $\frak{p}_D :=  \{ a \in D ~ | ~ v_D(a) > 0 \}$, which is principally generated by the uniformizer $\sqrt{p I} := \imath^{-1} \left( \begin{pmatrix} \sqrt{p} & 0 \\ 0 & -\sqrt{p} \end{pmatrix} \right) $.
The residue field $k_D := \cO_D / \frak{p}_D$ is isomorphic to $\FF_{p^2}$.

Let $D^\times$ (resp. $\cO_D^\times$) be the group of invertible elements of $D$ (resp. $\cO_D$). 
Via (\ref{quater embed}) $D^\times \subset \GL_2(\QQ_p(\sqrt{p}))$ is a locally $\QQ_p$-analytic subgroup of $\GL_2(\QQ_p(\sqrt{p}))$.
We define compact open normal (principal) subgroups \begin{eqnarray}\label{principal subgps}
U^n_D := 1+(\sqrt{p I})^n \cO_D, ~~ n \geq 1
\end{eqnarray} 
\begin{eqnarray}
\label{GL2 principal subgps}
U^n_{\GL_2} := 1+(\sqrt{p I})^n M_2(\ZZ_p(\sqrt{p})), ~~ n \geq 1
\end{eqnarray}
of $D^\times$ and $\GL_2(\ZZ_p(\sqrt{p}))$. We have
\[ D^{\times}=\mathcal{O}_D^{\times} \rtimes (\sqrt{p I})^{\mathbb{Z}}, \quad \mathcal{O}_D^{\times} / U_D^1 \cong \mathbb{F}_{p^2}^{\times}. \]

\begin{lem}\label{p power principal}
For any given positive integer $n$, every element of $U^{2(n+1)}_D$ (resp. $U^{2(n+1)}_{\GL_2}$) is the $p$-th power of an element of $U^{2n}_D$ (resp. $U^{2n}_{\GL_2}$).
Therefore we have \[ (U^{2n}_D)^p = U^{2(n+1)}_D, ~ (U^{2n}_{\GL_2})^p = U^{2(n+1)}_{\GL_2} . \]
\end{lem}
\begin{proof}
With the realization of $U^n_D$ via (\ref{quater embed}), the proof of \cite[Lem 5.1]{DDSMS99} applies here.
\end{proof}

\begin{cor}
The pro-$p$ groups $U^{2n}_D < U^{2n}_{\GL_2}$ are uniform for any positive integer $n$.
\end{cor}
\begin{proof}
For simplicity, we only give a proof for $U^2_D$.
For any given positive integer $n$, we have \[ U^{2n}_D / U^{2(n+1)} \simeq \cO_D / \frak{p}_D^2 \simeq \cO_D / (p) \] as the additive abelian group.
The group $U^2_D$ is finitely generated and powerful in the sense of \cite[Def 3.1]{DDSMS99} by \cite[Prop 1.14]{DDSMS99} and the above lemma.
By \cite[Thm 3.6]{DDSMS99} and the above lemma, for all $i \geq 1$, we have $P_i(U^2_D) = U^{2i}_D$ and
\[ |P_i(G): P_{i+1}(G)| = |\cO_D / (p)| = p^4. \]
\end{proof}

Let $Z_D$ be the center of $D^\times$ which is isomorphic to $\QQ_p^\times$. The subgroup $Z_D \cO_D^\times$ is of index 2 in $D^\times$.
Let $T_D$ be the pullback of diagonal torus via $\imath$ as a subgroup of $D^\times$ which is isomorphic to $\QQ_p(\sqrt{p})^\times$.
We define \[ T^n_D := T_D \cap U^n_D, ~~ n \geq 1. \]
These locally $\QQ_p$-analytic subgroups are isomorphic to $1+(\sqrt{pI})^n \cO_{\QQ_p(\sqrt{p})} \simeq \ZZ_p \oplus \ZZ_p$.
In particular, \[ T_D \simeq T^1_D \times \mu_{p-1} \times (\sqrt{pI})^\ZZ. \]

Consider the associated ($\ZZ_p$)-Lie algebra $\frak{g}\frak{l}_{2,\ZZ_p(\sqrt{p})}$ for $U^2_{\GL_2}$ according to (\ref{zp Lie}) via the log map
\[ \log: U^2_{\GL_2} \xrightarrow[\sim]{\log} p M_2(\ZZ_p(\sqrt{p})) \xrightarrow[\sim]{\cdot \frac{1}{p}} M_2(\ZZ_p(\sqrt{p})) \simeq \frak{g}\frak{l}_{2,\ZZ_p(\sqrt{p})}, \]
inducing the sub Lie algebra $\frak{g}_D$ associated to the subgroup $U^2_D$, identified with $\cO_D$
\begin{eqnarray}\label{log 2D} \log: U^2_D \xrightarrow[\sim]{\log} p\cO_D \xrightarrow[\sim]{\cdot \frac{1}{p}} \cO_D \simeq \frak{g}_D. \end{eqnarray}

\begin{lem}\label{log iso}
For any positive integer $n \geq 1$, the log map restricts to 
\[ \log: U^{2n}_D \xrightarrow{\sim} p^{n-1} \frak{g}_D, \]
\[ \log: U^{2n}_{\GL_2} \xrightarrow{\sim} p^{n-1} \frak{g}_{\GL_2}. \]
\end{lem}
\begin{proof}
By Lem \ref{p power principal}, any element $g \in U^{2n}_D$ (resp. $U^{2n}_{\GL_2}$) can be expressed as $g = g_2^{p^{n-1}}$ for some $g_2 \in U^2_D$ (resp. $U^2_{\GL_2}$).
We have $\log(g) = p^{n-1} \log(g_2) \in p^{n-1} \frak{g}_D$ (resp. $\frak{g}_{\GL_2}$).
\end{proof}

The $\ZZ_p$-action on $\frak{g}\frak{l}_{2,\ZZ_p(\sqrt{p})}$ obviously extends to a $\ZZ_p(\sqrt{p})$-action, equipping $\frak{g}\frak{l}_{2,\ZZ_p(\sqrt{p})}$ with the standard $\frak{g}\frak{l}_2$-Lie algebra structure over $\ZZ_p(\sqrt{p})$.
Therefore there is a natural inclusion
\begin{eqnarray}\label{sub Lie sqp} \frak{g}_D \hookrightarrow \frak{g}_D \otimes_{\ZZ_p} \ZZ_p(\sqrt{p}) \hookrightarrow \frak{g}\frak{l}_{2, \ZZ_p(\sqrt{p})}. \end{eqnarray}
Let $\frak{t}_0$ be the (sub) Lie algebra associated to the (sub) torus $T^2_D$. 
Let \[ e := \begin{pmatrix} 0 & 1 \\ 0 & 0 \end{pmatrix}, ~ f := \begin{pmatrix} 0 & 0 \\ 1 & 0 \end{pmatrix} \in \frak{g}\frak{l}_{2, \ZZ_p(\sqrt{p})}. \]
We use $\cO = \ZZ_p(\sqrt{p})$ for the rest of the section to ease notations.
We have the weight decomposition with respect to $\cO \cdot \frak{t}_0 := \frak{t}_0 \otimes_{\ZZ_p} \cO$:
\begin{eqnarray}\label{integral wd} \frak{g}_D \otimes_{\ZZ_p} \cO \subsetneq \cO \cdot e \oplus \cO \cdot \frak{t}_0 \oplus \cO \cdot f. \end{eqnarray}
Note that the second map of (\ref{sub Lie sqp}) is not surjective since $e$ and $f$ are not in $\frak{g}_D \otimes_{\ZZ_p} \cO$.

More explicitly, $\frak{g}_D$ is rank $4$ over $\ZZ_p$ with (via the log map)
\[ \frak{g}_D \simeq  \ZZ_p \cdot \begin{pmatrix} 1 & 0 \\ 0 & 1 \end{pmatrix} \oplus \ZZ_p \cdot \begin{pmatrix} \sqrt{p} & 0 \\ 0 & -\sqrt{p} \end{pmatrix} \oplus \ZZ_p \cdot \begin{pmatrix} 0 & \iota \\ 1 & 0 \end{pmatrix} \oplus \ZZ_p \begin{pmatrix} 0 & \iota \sqrt{p} \\ -\sqrt{p} & 0 \end{pmatrix}. \]
Let $h := \sqrt{pI} = \begin{pmatrix} \sqrt{p} & 0 \\ 0 & -\sqrt{p} \end{pmatrix}$, $w := \begin{pmatrix} 0 & \iota \\ 1 & 0 \end{pmatrix}$, $v := \sqrt{p} w = \begin{pmatrix} 0 & \iota \sqrt{p} \\ -\sqrt{p} & 0 \end{pmatrix}$. 
The Lie brackets are given by \begin{eqnarray}\label{lie brackets} [h,w] = 2v, ~ [h,v] = 2p w, ~ [w,v] = -2\iota h. \end{eqnarray}
With the extension to $\cO = \ZZ_p(\sqrt{p})$,
\begin{eqnarray}\label{alt basis} \frak{g}_D \otimes_{\ZZ_p} \cO & \simeq &  \cO \cdot I \oplus \cO \cdot \sqrt{pI} \oplus \cO \cdot w \oplus \cO \cdot (\sqrt{p} e).
\end{eqnarray}
The integral Casimir operator $\Delta \in U(\frak{g}_D)$ is given by
\begin{eqnarray}\label{casimir} \Delta := \iota h^2 + pw^2 - v^2. \end{eqnarray}

The inclusion \ref{sub Lie sqp} induces a morphism of corresponding universal enveloping algebras
\begin{eqnarray}\label{sub env sqp} U(\frak{g}_D) \to U_{\ZZ_p(\sqrt{p})}(\frak{g}_D \otimes_{\ZZ_p} \ZZ_p(\sqrt{p})) \to U_{\ZZ_p(\sqrt{p})}(\frak{g}\frak{l}_2), \end{eqnarray} where the subscript $\ZZ_p(\sqrt{p})$ denotes that the universal enveloping algebras are over $\ZZ_p(\sqrt{p})$.

\section{Some finite length modules of the completed enveloping algebra of $\frak{g}_D$}
In this section, we use notation from the previous sections.

Pick the torus \[ \frak{t}_D := \ZZ_p \cdot I \oplus \ZZ_p \cdot h \subset \frak{g}_D, \] with a character \[ \lambda : \frak{t}_D \to \ZZ_p. \]
Let $K$ be a finite extension of $\QQ_p$ with ring of integers $R$, uniformizer $\varpi_K$ such that $\varpi_K^s = p$ for $s \in \ZZ^+$.
Let $Z(\frak{g}_D)$ be the center of $\UgD$, $\chi : Z(\frak{g}_D) \to R$ be an infinitesimal character.
Letting \[ \frak{g}_R = \frac{1}{p} \frak{g}_D \otimes_{\ZZ_p} R, ~ \frak{g}_K = \frak{g}_R \otimes_{R} K, \] we define
\[ \wUgR := \varprojlim_{a} \left(\frac{U(\mathfrak{g}_R)}{p^{a} U(\mathfrak{g}_R)}\right), ~ \widehat{U(\frak{g}_K)} := \wUgR\otimes_{R} K .\]
Let $I(\lambda, \chi)$ be the left ideal of $\wUgR$ (resp. $U(\frak{g}_R)$, $U(\frak{g}_K)$) generated by \[ \{ t-\lambda(t), ~  z-\chi(z) ~ | ~ t \in \frak{t}_D, ~ z \in Z(\frak{g}_D) \}. \]
We define the quotient modules
\[ W^\circ_{\lambda, \chi} :=  \wUgR / I(\lambda, \chi), ~ W_{\lambda, \chi} := W^\circ_{\lambda, \chi} \otimes_{R} K. \]

\begin{lem}\label{W basis}
For any $n \in \ZZ^+$,
$W^\circ_{\lambda, \chi} / \varpi^n W^\circ_{\lambda, \chi}$ has a basis $\{w^k, w^k v| k \in \mathbb{N}\}$ over $R / \varpi^n$.
\end{lem}
\begin{proof}
By the PBW theorem, every element can be written as a $R / \varpi^n$-linear combination of $w^k v^i h^j$ for $k, i, j \in \mathbb{N}$. 
All the monomial terms such that $j \geq 1$ are reduced to $j = 0$ terms since $h \in \frak{t}_D$.
All the monomial terms such that $j = 0, ~ i \geq 2$ are reduced to $j = 0, ~ i \leq 1$ terms as $v$ is quadratic in $\Delta = \iota h^2 + pw^2 - v^2 \in Z(\frak{g}_D)$.
The center $Z(\frak{g}_D)$ is a polynomial ring of $\Delta$ with coefficients in $\ZZ_p$.
One can verify that any non-zero $R / \varpi^n$-linear combination of $\{w^k, w^k v| k \in \mathbb{N}\}$ does not belong to $R / \varpi^n$-span of $U(\mathfrak{g}_R) (h-\lambda(h)) + U(\mathfrak{g}_R) (\Delta-\chi(\Delta))$.
\end{proof}
\begin{cor}\label{W basis seris}
  $W_{\lambda, \chi}$ is identified with
  \[ \left\{ (a_i, b_i)_{i \in \NN} \in \prod_{i=0}^\infty K \oplus \prod_{i=0}^\infty K \bigg| \|a_i\|_p \to 0, ~ \|b_i\|_p \to 0 \right\} \xrightarrow{\sim} W_{\lambda, \chi} \] by sending $\displaystyle (a_i, b_i)_{i \in \NN} \mapsto \sum_{i=0}^\infty a_i w^i + \sum_{i=0}^\infty b_i w^i v$.
\end{cor}

We consider the two dimensional general linear $\ZZ_p(\sqrt{p})$-Lie algebra $\frak{g}\frak{l}_2(\ZZ_p(\sqrt{p}))$ over $\ZZ_p(\sqrt{p})$.
To ease some notation, we define $\frak{g}\frak{l}_{2,\sqrt{p}}$ to be
\[ \frak{g}\frak{l}_{2,\sqrt{p}} := \frak{g}\frak{l}_2(\ZZ_p(\sqrt{p})) \otimes_{\ZZ_p(\sqrt{p})} R. \]
Let \[ \frak{t}_2 := \ZZ_p(\sqrt{p}) \oplus \ZZ_p(\sqrt{p}) \subset \frak{g}\frak{l}_2(\ZZ_p(\sqrt{p})) \] be the diagonal torus. 
Let \[ \tilde{\lambda} : \frak{t}_2 \to \ZZ_p, \] \[ \tilde{\chi} : Z(\frak{g}\frak{l}_{2,\sqrt{p}}) \to R \] be characters whose restrictions on $\frak{t}_D$ and $Z(\frak{g}_D)$ coincide with $\lambda$ and $\chi$.
Similarly, let $I(\tilde{\lambda}, \tilde{\chi})$ be the left ideal of $\widehat{U(\frak{g}\frak{l}_{2,\sqrt{p}})}$
generated by \[ \{ t-\tilde{\lambda}(t), ~  z-\tilde{\chi}(z) ~ | ~ t \in \frak{t}_2, ~ z \in Z(\frak{g}\frak{l}_{2,\sqrt{p}}) \}. \]
We denote the quotient modeules
\[ W^\circ_{\tilde{\lambda}, \tilde{\chi}} :=  \widehat{U(\frak{g}\frak{l}_{2,\sqrt{p}})} / I(\tilde{\lambda}, \tilde{\chi}), ~ W_{\tilde{\lambda}, \tilde{\chi}} := W^\circ_{\tilde{\lambda}, \tilde{\chi}} \otimes_{R} K. \]

The morphism (\ref{sub env sqp}) extends to their completions
\begin{eqnarray}\label{sub env sqp comp}
  \wUgR \to \widehat{U(\frak{g}\frak{l}_{2,\sqrt{p}})}.
\end{eqnarray}

\begin{prop}
  There is an embedding \[ W_{\lambda, \chi} \hookrightarrow W_{\tilde{\lambda}, \tilde{\chi}} \] as $\wUgR$-modules induced by morphism (\ref{sub env sqp comp}).
\end{prop}
\begin{proof}
Let $h_0 := \begin{pmatrix} 1 & 0 \\ 0 & -1 \end{pmatrix}$, $w_0 := \begin{pmatrix} 0 & \iota \\ 1 & 0 \end{pmatrix}$, $v_0 := \begin{pmatrix} 0 & \iota \\ -1 & 0 \end{pmatrix}$.
We have \[ \frak{g}\frak{l}_2(\ZZ_p(\sqrt{p})) \cong \ZZ_p(\sqrt{p}) \cdot I \oplus \ZZ_p(\sqrt{p}) \cdot h_0 \oplus \ZZ_p(\sqrt{p}) \cdot w_0 \oplus \ZZ_p(\sqrt{p}) \cdot v_0. \]
The element \[ \Delta_0 := \iota h_0^2 + w_0^2 - v_0^2 \] is in the center of $\frak{g}\frak{l}_2(\ZZ_p(\sqrt{p}))$. 
Similar to Lem \ref{W basis seris}, $W^\circ_{\tilde{\lambda}, \tilde{\chi}} / \varpi^n W^\circ_{\tilde{\lambda}, \tilde{\chi}}$ has a basis $\{w_0^k, w_0^k v_0| k \in \mathbb{N}\}$ over $R / \varpi^n$ and $W_{\tilde{\lambda}, \tilde{\chi}}$ can be identified with
\[ \left\{ (a_i, b_i)_{i \in \NN} \in \prod_{i=0}^\infty K \oplus \prod_{i=0}^\infty K \bigg| \|a_i\|_p \to 0, ~ \|b_i\|_p \to 0 \right\} \xrightarrow{\sim} W_{\lambda, \chi} \] by sending $\displaystyle (a_i, b_i)_{i \in \NN} \mapsto \sum_{i=0}^\infty a_i w_0^i + \sum_{i=0}^\infty b_i w_0^i v_0$.
The map $W_{\lambda, \chi} \hookrightarrow W_{\tilde{\lambda}, \tilde{\chi}}$ can be concretely described as 
\[ \sum_{i=0}^\infty a_i w^i + \sum_{i=0}^\infty b_i w^i v \mapsto \sum_{i=0}^\infty a_i w_0^i + \sum_{i=0}^\infty \sqrt{p} b_i w_0^i v_0, \] which is obviously injective.
\end{proof}

\begin{cor}\label{uniqueness ev}
  An eigenspace of $h$ on $W_{\lambda, \chi}$ of any eigenvalue has dimension 1.
\end{cor}

\begin{proof}
  Consider
  \[ e_0 := \begin{pmatrix} 0 & 1 \\ 0 & 0 \end{pmatrix} = \frac{1}{2\iota} (w_0 + v_0), \]
  \[ f_0 := \begin{pmatrix} 0 & 0 \\ 1 & 0 \end{pmatrix} = \frac{1}{2} (w_0 - v_0). \]
  The module $W^\circ_{\tilde{\lambda}, \tilde{\chi}}$ has another basis $\{1, e_0^k, f_0^k| k \in \mathbb{N}\}$ which is an eigenbasis for $h$ such that
  \[ h \cdot e_0^k = \lambda(h) + 2k, h \cdot f_0^k = \lambda(h) - 2k. \]
  Any eigenspace of $h$ on $W^\circ_{\tilde{\lambda}, \tilde{\chi}}$ has dimension 1. The same holds for its $h$ invariant subspace $W_{\lambda, \chi}$.
\end{proof}

\begin{definition}\label{infinitesimal genericity}
We define $\lambda$ as \emph{generic} if $\lambda$ does not appear in any absolutely irreducible finite-dimensional representation of $\frak{g}_D$.
\end{definition}

Suppose $\frak{g}$ is the $\ZZ_p$-Lie algebra associated to a uniform pro-$p$ group $G$ via (\ref{zp Lie}). 
The Lie algebra associated to $G^{p^n}$ is naturally identified with \[ p^n \frak{g} \subset \frak{g} \] for each natural number $n \geq 0$.
Let \[ \wUgnK := \widehat{U(p^n\frak{g})_K} \] be the $K$ coefficient completed universal enveloping algebra associated to $p^n \frak{g}$.
There is a natural inclusion \begin{eqnarray} \wUgnK \hookrightarrow \widehat{U(\frak{g})_{n', K}} \end{eqnarray} whenever $n \geq n'$, with the faster coefficient convergent rate for the former subring.
In particular, we have \begin{eqnarray}\label{env alg inc} \wUgnK \hookrightarrow \wUgK. \end{eqnarray}

Consider the adjoint action $G \to \GL(\frak{g}_K)$.
Let $T \leq G$ be the stabilizer subgroup of $\frak{t}$.
We further suppose the $\ZZ_p$-Lie algebra associated to $T$ via (\ref{zp Lie}) is a $\ZZ_p$-lattice in $\frak{t}$.
Given an element $g \in G$, let \begin{eqnarray*} g \circ \lambda : g^{-1}\frak{t}g & \to & K^\times \\ t & \mapsto & \lambda(gtg^{-1}) \end{eqnarray*} be the conjugation of $\lambda$ by $g$. 
For notation, similarly let $I(g \circ \lambda, \chi)$ be the left ideal of $\wUgnK$ generated by \[ \{ t-g \circ \lambda(t), ~  z-\chi(z) ~ | ~ t \in p^n g^{-1}\frak{t}g, ~ z \in Z(p^n \frak{g}) \}, ~ \mathrm{and ~ let} \] 
\[ W^{n, g}_{\lambda, \chi} := \wUgnK / I(g \circ \lambda, \chi) \] be a weight module of $\wUgnK$ with respect to $g^{-1} \frak{t} g$ as well as its natural integral model $W^{n, g, \circ}_{\lambda, \chi}$.

\begin{rem}\label{power convergence rate}
  Similar to Cor \ref{W basis seris}, $W^{n,\mathrm{id}}_{\lambda, \chi}$ is identified with
  \begin{eqnarray*} \left\{ (a_i, b_i)_{i \in \NN} \in \prod_{i=0}^\infty K \oplus \prod_{i=0}^\infty K \bigg| p^{ni} \cdot \|a_i\|_p \to 0, ~ p^{ni} \cdot \|b_i\|_p \to 0 \right\} \xrightarrow{\sim} W^{n,\mathrm{id}}_{\lambda, \chi} \end{eqnarray*}
  sending $\displaystyle (a_i, b_i)_{i \in \NN} \mapsto \sum_{i=0}^\infty a_i w^i + \sum_{i=0}^\infty b_i w^i v$.

  The natural inclusion (\ref{env alg inc}) therefore induces injections \begin{eqnarray}\label{inclusion of Ws} W^{n,\mathrm{id}}_{\lambda, \chi} \hookrightarrow W_{\lambda, \chi}, ~~~  W^{n,\mathrm{id}, \circ}_{\lambda, \chi} \hookrightarrow W^\circ_{\lambda, \chi}. \end{eqnarray}

  For a general $g \in G$, we can embed $W^{n, g}_{\lambda, \chi}$ into $W_{\lambda, \chi}$ via
  \begin{eqnarray}\label{g inclusion of Ws}
    W^{n, g}_{\lambda, \chi} & \hookrightarrow & W_{\lambda, \chi} \\
    w & \mapsto & w \cdot g^{-1}. \nonumber
  \end{eqnarray}
\end{rem}

For the rest of the section, we take $G := U^2_D$ in \S \ref{padic quaternion}.
Suppose $\sqrt{p} \in K$.
Recall the weight decomposition (\ref{integral wd}) along with the $\ZZ_p(\sqrt{p})$ (therefore $\cO_K$) basis (\ref{alt basis}).

We need an elementary lemma for our next key proposition.
\begin{lem}\label{factorial pval}
  For $N > 0, n \geq 0$ and $0 < a < p$ such that $N + an \leq p^k$, we have 
  \[ v_p(\prod_{i=0}^{i=n} (N+ia)) \leq k + \frac{n}{p-1}. \]
\end{lem}
\begin{proof}
Suppose $v_p(N+n_0 a) = \max_{0 \leq i \leq n} (N+ia)$, we can separately bound
\[ v_p(\prod_{i=0}^{i<n_0} (N+ia)) \leq \frac{1}{p-1} \cdot |\{0 \leq i < n_0\}| \]
\[ v_p(\prod_{i=n_0+1}^{i=n} (N+ia)) \leq \frac{1}{p-1} \cdot |\{n_0 < i \leq n\}|, \]
putting them together, we have
\[ v_p(\prod_{i=0}^{i=n} (N+ia)) \leq v_p(N+n_0 a) + \frac{1}{p-1} \cdot (|\{0 \leq i < n_0\}| + |\{n_0 < i \leq n\}|) \leq k + \frac{n}{p-1}. \]
\end{proof}

\begin{prop}\label{g not in id}
 If $g \notin T \cdot G^{p^n}$, the image of $g$ in $W_{\lambda, \chi}$ does not lie in $W^{n,\mathrm{id}}_{\lambda, \chi}$.
\end{prop}

\begin{proof}
By \cite[Thm 4.9]{DDSMS99} together with Lem \ref{log iso}, we express $g$ as 
  \[ g = \exp(c_w \cdot w) \cdot \exp(c_v \cdot v) \cdot \exp(c_h \cdot h), \]
for $c_w, c_v, c_h \in p \cdot \ZZ_p$. As the image of $\exp(c_h \cdot h)$ equals to a constant, it suffices to consider the  $c_h = 0$ case.

For the generic case $c_w \cdot c_v \neq 0$, we write \[ c_w = p^{n_w} c_w^\ast, ~ c_v = p^{n_v} c_v^\ast ~~ \for \] \[ n_w \geq 1, ~ n_v \geq 1, ~ c_w^\ast, c_v^\ast \in \ZZ_p^\ast. \]
By \cite[Lem 2.5]{DDSMS99} and \cite[Cor 2.8]{DDSMS99}, the condition $g \notin T \cdot G^{p^n} = G^{p^n} \cdot T$ is equivalent to either $n_w$ or $n_v$ is at most $n$.

We further decompose $g = g_w \cdot g_v$ for \[ g_w := \exp(c_w \cdot w) = \sum_{i=0}^\infty \frac{1}{i!} c_w^i w^i \] 
\[ g_v := \exp(c_v \cdot v) = \sum_{j=0}^\infty \frac{1}{j!} c_v^j v^j \]

We express their images in $W_{\lambda, \chi}$ by the following convergent power series (Lemma \ref{W basis})
\[ g_w = \sum_{m=0}^\infty c_{g_w, m} w^m, ~~ g_v = \sum_{m=0}^\infty c_{0, g_v, m} w^m + \sum_{m=0}^\infty c_{1, g_v, m} w^m v. \]

Since $p$-adic valuation of $v_p(m!) \sim \frac{m}{p-1}$ as $m \to \infty$, we have \[v_p(c_{g_w, m}) \sim m(n_w - \frac{1}{p-1}).\]

With the casimir operator relation \ref{casimir}, $v^2 = pw^2 + \imath h^2 - \chi(\Delta)$. To ease notation, let $x := \imath h^2 - \chi(\Delta)$. We unfold $g_v$ as
\[ g_v = \sum_{j=0}^\infty \frac{1}{(2j)!} c_v^{2j} (pw^2 +x)^j + \sum_{j=0}^\infty \frac{1}{(2j+1)!} c_v^{2j+1} (pw^2 +x)^j v. \]

By Lie brackets (\ref{lie brackets}), we can express the second part of the sum as
\[ \sum_{j=0}^\infty \frac{1}{(2j+1)!} c_v^{2j+1} (pw^2 +x)^j v = \sum_{m=0}^\infty c_{0, g_v, 2m+1} w^{2m+1} + \sum_{m=0}^\infty c_{1, g_v, 2m} w^{2m} v. \]

Therefore, the first part of the sum only consists of monomials of even degrees
\[ \sum_{j=0}^\infty \frac{1}{(2j)!} c_v^{2j} (pw^2 +x)^j = \sum_{m=0}^\infty c_{0, g_v, 2m} w^{2m}. \]

For this part, we can regard $x$ as a scaler in $\cO_K$. 

Particularly when $m = \frac{p^k + 1}{2}$ for $k \geq 1$, 
\begin{eqnarray}\label{k power} v_p(c_{0, g_v, p^k + 1}) = v_p\big(\frac{1}{(p^k)!} c_v^{p^k + 1} p^{\frac{p^k + 1}{2}}\big). \end{eqnarray}

For general $m$, by the binomial theorem, we have
\begin{eqnarray}\label{min pval} v_p(c_{0, g_v, 2m}) \geq \min_{N \geq m} v_p\big(\frac{1}{(2N)!} c_v^{2N} p^m \binom{N}{m}\big). \end{eqnarray}

Therefore, we have a uniform lower bound
\[ v_p(c_{0, g_v, m}) \geq O((n_v+\frac{1}{2}-\frac{1}{p-1})m), ~~ v_p(c_{1, g_v, m}) \geq O((n_v+\frac{1}{2}-\frac{1}{p-1})m). \]

We expand $g = g_w \cdot g_v$ as 
\[ g = \sum_{m=0}^\infty c_{0, g, m} w^m + \sum_{n=0}^\infty c_{1, g, m} w^m v, ~~ \mathrm{where}\]
\[ c_{l, g, m} = \sum_{i=0}^m c_{g_w, m-i} c_{l, g_v, i}, ~~\for ~~ l = 0, 1, \]
and discuss two cases:
\begin{itemize}
  \item If $n_v < n_w$, we claim 
  \[ v_p(c_{0,g,p^k+1}) = v_p\big(\frac{1}{(p^k)!} c_v^{p^k + 1} p^{\frac{p^k + 1}{2}}\big) \sim (p^k+1)(n_v + \frac{1}{2} - \frac{1}{p-1}), \] for any arbitrarily large $k$.

  For $0 < i < \frac{p^k + 1}{2}$, we have
  \begin{eqnarray}\label{comb pk}
    v_p(\binom{p^k+1}{2i}) \geq k.
  \end{eqnarray}
  We can therefore bound for $0 < i < \frac{p^k + 1}{2}$
  \begin{eqnarray*} v_p(c_{g_w,p^k+1-2i} c_{0,g_v, 2i}) & = & v_p(\frac{1}{(p^k+1-2i)!} c_w^{p^k+1-2i}) + v_p(c_{0,g_v, 2i}) \\ 
  \ref{min pval} & \geq & v_p(\frac{1}{(p^k+1-2i)!} c_w^{p^k+1-2i}) + \min_{N \geq i} v_p\big(\frac{1}{(2N)!} c_v^{2N} p^{i} \binom{N}{i}\big) \\
  & \geq & v_p(c_w^{p^k + 1 - 2i} * c_v^{2i} * p^i * \frac{1}{(p^k+1)!}) + v_p(\binom{p^k+1}{2i}) \\ 
  & + & \min_{N \geq i} v_p(c_v^{2N-2i} * \frac{1}{(N-i)!} * \frac{(2i-1)!!}{(2N-1)!!}) \\
  \ref{comb pk} & > & v_p\big(\frac{1}{(p^k)!} c_v^{p^k + 1} p^{\frac{p^k + 1}{2}}\big) + k + (2N-2i)n_v - \frac{N-i}{p-1} \\
  & + & \min_{N \geq i} v_p(\frac{1}{(2i+1)\cdots(2N-1)}) \\
  \mathrm{Lem} ~ \ref{factorial pval} & \geq & v_p\big(\frac{1}{(p^k)!} c_v^{p^k + 1} p^{\frac{p^k + 1}{2}}\big) + (2N-2i)(n_v - \frac{1}{p-1}) \\
  & \geq & v_p\big(\frac{1}{(p^k)!} c_v^{p^k + 1} p^{\frac{p^k + 1}{2}}\big). \end{eqnarray*}
  When $i = 0$,
  \[ v_p(c_{g_w,p^k + 1} c_{0,g_v, 0}) = v_p\big(\frac{1}{(p^k)!} c_w^{p^k + 1}\big) > v_p\big(\frac{1}{(p^k)!} c_v^{p^k + 1} p^{\frac{p^k + 1}{2}}\big). \]
  By (\ref{k power}), $v_p(c_{g_w,0} c_{0,g_v, p^k + 1}) = v_p\big(\frac{1}{(p^k)!} c_v^{p^k + 1} p^{\frac{p^k + 1}{2}}\big)$ is the unique term with smallest $p$-valuation in the sum $\sum_{i=0}^m c_{g_w, m-i} c_{0, g_v, i}$, we conclude
  \[ v_p(c_{0,g,p^k+1}) = v_p\big(\frac{1}{(p^k)!} c_v^{p^k + 1} p^{\frac{p^k + 1}{2}}\big) \sim (p^k+1)(n_v + \frac{1}{2} - \frac{1}{p-1}). \]
  \item If $n_v \geq n_w$, we can similarly conclude \[ v_p(c_{0,g,p^k+1}) = v_p\big(\frac{1}{(p^k)!} c_w^{p^k + 1}\big) \sim (p^k+1)(n_w - \frac{1}{p-1}), \] for any arbitrarily large $k$. 
\end{itemize}

In either case, $v_p(c_{0,g,p^k+1}) \ll (p^k+1) n$. Therefore, $g = \sum_{m=0}^\infty c_{0, g, m} w^m + \sum_{m=0}^\infty c_{1, g, m} w^m v$ does not lie in $W^{n,\mathrm{id}}_{\lambda, \chi}$ by the coefficients convergence description of Rem \ref{power convergence rate}.
\end{proof}

\begin{theorem}\label{non isomorphic}
Two $\wUgnK$ weight modules $W^{n, g}_{\lambda, \chi}$ and $W^{n, g'}_{\lambda, \chi}$ are isomorphic to each other if and only if $g^{-1}g' \in T \cdot G^{p^n}$.
\end{theorem}
\begin{proof}
  We may assume $g' = \mathrm{id}$ is the identity element. If $g \in T \cdot G^{p^n}$ such that $g = t_1 g_1$ for $t_1 \in T, ~ g_1 \in G^{p^n}$, it is direct to check \begin{eqnarray*} W^{n, \mathrm{id}} & \to & W^{n, g} \\ m & \mapsto & m g_1 \end{eqnarray*} is a $\wUgnK$-equivariant isomorphism as $T$ is commutative and $g^{-1} \circ \lambda = g_1^{-1} \circ \lambda$.

  If $g \notin T \cdot G^{p^n}$, we consider the embeddings (\ref{inclusion of Ws}, \ref{g inclusion of Ws})
  \[ W^{n,\mathrm{id}}_{\lambda, \chi} \hookrightarrow W_{\lambda, \chi}, ~~~  W^{n, g}_{\lambda, \chi} \hookrightarrow W_{\lambda, \chi}. \]

  The image of $1 \in \wUgnK$ in $W^{n,\mathrm{id}}_{\lambda, \chi} \hookrightarrow W_{\lambda, \chi}$ is a unique eigenvector in $W_{\lambda, \chi}$ up to a scaler by Cor \ref{uniqueness ev}. 

  If $W^{n, g}_{\lambda, \chi}$ and $W^{n, \mathrm{id}}_{\lambda, \chi}$ are isomorphic, $1$ must also lie in $W^{n, g}_{\lambda, \chi}$ via embedding (\ref{g inclusion of Ws}). 
  Equivalently, the group element $g$ as an element in $W_{\lambda, \chi}$ must lie in $W^{n, g}_{\lambda, \chi}$: we have the following expression
  \[ (u+t_g) g^{-1} = 1 + t, ~~ \mathrm{where} ~ u \in \wUgnK, ~ t \in I(\lambda, \chi), ~ t_g \in I(g\circ\lambda, \chi). \]
  Multiplying $g$ on the left to both sides, we get
  \[ g = gug^{-1} + gt_g g^{-1} - gt.\]
  As $G$ is uniform pro-$p$, $G^{p^n}$ is normal in $G$, $g$ stablizes $\wUgnK$. Note that $gI(g\circ\lambda, \chi)g^{-1} = I(g\circ\lambda, \chi)g^{-1} = I(\lambda, \chi)$. We conclude $g$ lies in $W^{n,\mathrm{id}}_{\lambda, \chi}$, contradicting Prop \ref{g not in id}.
\end{proof}

\section{Construction of the representations}\label{Con}
We refer to \S \ref{padic quaternion} for notations.
Given a continuous (or equivalently, locally analytic) character $\lambda : T_D \to \QQ_p^\times$, we construct a Banach representation of $D^\times$
\[ V_\lambda := \{ f \in C(D^\times, \QQ_p) ~ | ~ f(gt) = \lambda(t)f(g) ~~ \for ~ \forall ~ t \in T_D,  ~ g \in D^\times \}\]

\begin{lem}\label{V lambda adm}
For any continuous character $\lambda : T_D \to \QQ_p^\times$, the Banach representation $V_\lambda$ is admissible.
\end{lem}
\begin{proof}
We have $\cO_D^\times T_D = D^\times \simeq \cO_D^\times \times (\sqrt{p})^\ZZ$.
As a representation of $\cO_D^\times$,
\[ V_\lambda \simeq \{ f \in C(\cO_D^\times, \QQ_p) ~ | ~ f(gt) = \lambda(t)f(g) ~~ \for ~ \forall ~ t \in T^1_D \times \mu_{p-1},  ~ g \in \cO_D^\times \}. \]
This is a closed sub representation of $C(\cO_D^\times, \QQ_p)$, hence admissible.
\end{proof}

Similarly, we define the admissible Banach representation of $U^2_D$
\[ V^2_\lambda := \{ f \in C(U^2_D, \QQ_p) ~ | ~ f(gt) = \lambda(t)f(g) ~~ \for ~ \forall ~ t \in T^2_D,  ~ g \in D^\times \}. \]

For $T^2_D$, we may pick two topological generators \[ t_1 = 1+pI \in Z_D \cap T_D, ~~ t_2 = 1+p\cdot \sqrt{pI} \in T_D \backslash Z_D \cap T_D. \]

Let $\lambda_1 := \lambda(t_1), \lambda_2 := \lambda(t_2)$. We have the following left exact sequence of admissible Banach representations
\begin{eqnarray*} 0 \to V^2_\lambda \to C(U^2_D, \QQ_p) & \to & C(U^2_D, \QQ_p) \oplus C(U^2_D, \QQ_p) \\ f & \mapsto & (f( \cdot t_1) - \lambda_1 f, f( \cdot t_2) - \lambda_2 f) \end{eqnarray*}
The representation $V^2_\lambda$ is dual to the cyclic Iwasawa module
\[ I^{2,\lambda} := \QQ_p[[U^2_D]] / \QQ_p[[U^2_D]] \cdot (t_1-\lambda_1, t_2-\lambda_2) \] as the above left exact sequence is dual to the following right exact sequence
\begin{eqnarray*} \QQ_p[[U^2_D]] \oplus \QQ_p[[U^2_D]] & \to & \QQ_p[[U^2_D]] \to I^{2,\lambda} \to 0 \\ (x,y) & \mapsto & x (t_1-\lambda_1) + y (t_2-\lambda_2). \end{eqnarray*}
We may base change $I^{2,\lambda}$ via (\ref{ST dis})
\[ D^{2,\lambda} := D(U^2_D, \QQ_p) \otimes_{\QQ_p[[U^2_D]]} I^{2,\lambda}, \]
\[ D^{2,\lambda}_r := D_r(U^2_D, \QQ_p) \otimes_{\QQ_p[[U^2_D]]} I^{2,\lambda}. \]
Let $\chi : Z(\frak{g}_K) \to K$ be an infinitesimal character for a finite extension $K / \QQ_p$, and let $D^2_r$ be $D_r(U^2_D, \QQ_p)$. We define 
\[ D^{2,\chi} := D(U^2_D, \QQ_p) \otimes_{Z(\frak{g}_K), \chi} K, \]
\[ D^{2,\chi}_r := D^2_r \otimes_{Z(\frak{g}_K), \chi} K, \]
\[ D^{2,\lambda, \chi} := (D^{2,\lambda} \otimes_{\QQ_p} K) \otimes_{Z(\frak{g}_K), \chi} K \cong D^{2,\chi} \otimes_{\QQ_p[[U^2_D]]} I^{2,\lambda}, \]
\[ D^{2,\lambda, \chi}_r := (D^{2,\lambda}_r \otimes_{\QQ_p} K) \otimes_{Z(\frak{g}_K), \chi} K \cong D^{2,\chi}_r \otimes_{\QQ_p[[U^2_D]]} I^{2,\lambda}.\]

\begin{theorem}\label{length 2}
If $\lambda$ is generic in the sense of Def (\ref{infinitesimal genericity}) and $r = \sqrt[p^n]{1 / p}$ for $n \in \ZZ_{\geq 1}$, the length of $D^{2,\lambda,\chi}_r$ (resp. $W^{n,\mathrm{id}}_{\lambda, \chi}$) as a $D^2_r$ (resp. $\wUgnK$) module is at most 2.
\end{theorem}
\begin{proof}
To ease notations, we use $A$ to denote $D^2_r$ (resp. $\wUgnK$) and $M$ to denote $D^{2,\lambda,\chi}_r$ (resp. $W^{n,\mathrm{id}}_{\lambda, \chi}$).

As $D^{2,\lambda,\chi}_r$ is complete with respect to the norm $\| ~ \|_r$ (\ref{r norm}) and $W^{n,\mathrm{id}}_{\lambda, \chi}$ is $p$-adically complete, we will illustrate that $A$ is a complete doubly filtered $K$-algebra in the sense of \cite[Def 3.1]{AW13} in the following, therefore induces a filtration on $M$.

We can assume all $p^n$-th roots of $p$ are contained in $K$. Otherwise, we can always pass to a larger field to prove our statement. Let $k$ be the residue field of $K$. We pick a $p^n$-th root of $p$ and denote it by $p^{\frac{1}{p^n}}$. We take $A$ as $D^2_r$ to illustrate the general case for the rest of the proof.

Any $a \in A$ can be expressed as 
\[ a = \sum_{\alpha \in \mathbb{N}^{d}} \lambda_{\alpha} \mathbf{b}^{\alpha} \mid \lambda_\alpha \in K, ~ \sup_{\alpha \in \mathbb{N}^{d}} |\lambda_\alpha| < \infty \]
via (\ref{Dr series}), where 
\[ \mathbf{b}^{\alpha} = \mathbb{I}^{\alpha_I} \cdot H^{\alpha_h} \cdot W^{\alpha_w} \cdot V^{\alpha_v}, \]
\[ \mathbb{I} := \frac{1}{p^{\frac{1}{p^n}}} (\log^{-1}(I)-1), ~ H := \frac{1}{p^{\frac{1}{p^n}}} (\log^{-1}(h)-1), \] 
\[ W :=\frac{1}{p^{\frac{1}{p^n}}} (\log^{-1}(w)-1), ~ V := \frac{1}{p^{\frac{1}{p^n}}} (\log^{-1}(v)-1) \] such that the log map is defined by (\ref{log 2D}) and the elements $I, h, w, v$ are defined above (\ref{lie brackets}).
For $\| ~ \|_r$ on $D^2_r$, the filtration is defined by
\[ F_{s}^{r} A:=\left\{ a \in A:\|a\|_{r} \leq p^{s} \right\}. \]

With respect to this filtration, the augmentation ideal $\frak{m}$ of $\gr_0(A)$ as the kernel of $\gr_0(A) \twoheadrightarrow k$ is generated by $\mathbb{I}, H, W, V$. A basis of $\frak{m}^l / \frak{m}^{l+1}$ is given by \[ \{\mathbf{b}^{\alpha} \big| |\alpha| = l \}. \]
Let $J$ be the augmentation ideal of $R[[U^2_D]]$. We apply Lazard's famous theorem on the structure of the graded ring of completed group algebra to the uniform pro-$p$ group $U^2_D$ 
\[ \gr_J(R[[U^2_D]]) \cong k[\tilde{p}, p^{\frac{1}{p^n}}\mathbb{I}, p^{\frac{1}{p^n}}H, p^{\frac{1}{p^n}}W, p^{\frac{1}{p^n}}V], \]
where $k[\tilde{p}, p^{\frac{1}{p^n}}\mathbb{I}, p^{\frac{1}{p^n}}H, p^{\frac{1}{p^n}}W, p^{\frac{1}{p^n}}V]$ is commutative. This implies 
\[ [p^{\frac{l_1}{p^n}} \mathbf{b}^{\alpha_1}, p^{\frac{l_2}{p^n}} \mathbf{b}^{\alpha_2}] \in J^{l_1 + l_2 + 1} ~ \mathrm{for} ~ \forall \ \alpha_1, \alpha_2 \in \mathbb{N}^{d} \big| |\alpha_1| = l_1, |\alpha_2| = l_2. \]
Note that $\|a\|_r \leq r^l$ for any $a \in J^l$, we get
\[ \|[\mathbf{b}^{\alpha_1}, \mathbf{b}^{\alpha_2}]\|_r \leq r < 1, ~ \forall \alpha_1, \alpha_2 \in \mathbb{N}^{d} \] or equivalently,
\[ [\mathbf{b}^{\alpha_1}, \mathbf{b}^{\alpha_2}] \in F^r_{<0} A, ~ \forall \alpha_1, \alpha_2 \in \mathbb{N}^{d}. \]
We conclude $\gr_0(A)$, and therefore $\Gr(A)$ are commutative,
\[ \Gr (A) \cong \gr_0(A) \cong k[\mathbb{I}, H,W,V]. \]
By \ref{casimir}, we have
\[ \Gr (M) \cong k[W,V] / (V^2). \]

We can choose good double filtrations on $M$'s subquotients as in \cite[Prop 3.3]{AW13}.
Subquotients of $M$ induce subquotients on $\Gr(M)$.
If $M$ has length at least 3, one of its subquotients $M_{sq}$ satisfies
\[ \dim_{k} \Gr (M_{sq}) < \infty. \]
By \cite[Prop 9.1, (a)]{AW13}, $M_{sq}$ itself is finite dimensional over $K$, contradicting to the genericity assumption on $\chi$ (Def \ref{infinitesimal genericity}).
\end{proof}

To ease notation, we take $G := U^2_D$ in \S \ref{padic quaternion} for the following 3 propositions: Prop \ref{Ugn decomposition}, Prop \ref{M decomposition}, Prop \ref{restriction invariance}.
\begin{prop}\label{Ugn decomposition}
  If $r = \sqrt[p^n]{1 / p}$ for $n \in \ZZ_{\geq 1}$, $D^{2,\lambda,\chi}_r$ decomposes as 
  \begin{eqnarray}\label{D2 decomposition} D^{2,\lambda,\chi}_r \cong \bigoplus_{g \in G / T \cdot G^{p^n}} W^{n, g}_{\lambda, \chi}\end{eqnarray}
  as a left $\wUgnK$-module.
\end{prop}
\begin{proof}
Let $\widehat{U(\frak{g})_{n,K}^{\chi}}$ be the quotient of $\wUgnK$ by $\chi$.
Note that the arguments and references of \cite[Prop 5.3]{Fu24} apply to our group $U^2_D$ (no need for the reductive group to be split), we have
\[ D^{2,\chi}_r \cong \bigoplus_{h \in G / G^{p^n}} \widehat{U(\frak{g})_{n,K}^{\chi}} \cdot h^{-1} \]
as a left $\wUgnK$-module.

For any $g \in G / T \cdot G^{p^n}$, we define a map 
\begin{eqnarray*}
  W^{n, g}_{\lambda, \chi} & \to & D^{2,\lambda,\chi}_r \\
  w & \mapsto & w \cdot g^{-1}
\end{eqnarray*}
similar to (\ref{g inclusion of Ws}). We, therefore obtain a map
\[ \bigoplus_{g \in G / T \cdot G^{p^n}} W^{n, g}_{\lambda, \chi} \to D^{2,\lambda,\chi}_r. \]

For each $g \in G / T \cdot G^{p^n}$, we pick distinct representatives $h_1,\cdots,h_t \in G / G^{p^n}$ lifting $g$ such that $|T \cdot G^{p^n} / G^{p^n}| = t$ and define a $\widehat{U(\frak{g})_{n,K}^{\chi}}$-equivariant map
\begin{eqnarray*}
  \bigoplus_{i=1}^t \widehat{U(\frak{g})_{n,K}^{\chi}} \cdot h_i^{-1} & \to & W^{n, g}_{\lambda, \chi} \\
  \sum_{i=1}^{t} w_i h_i^{-1} & \mapsto & \sum_{i=1}^{t} \lambda(g h_i^{-1}) w_i,
\end{eqnarray*}
extending to \[ D^{2,\chi}_r \cong \bigoplus_{g \in G / T \cdot G^{p^n}} \big(\bigoplus_{i=1}^t \widehat{U(\frak{g})_{n,K}^{\chi}} \cdot h_i^{-1}\big) \to \bigoplus_{g \in G / T \cdot G^{p^n}} W^{n, g}_{\lambda, \chi}. \]
It is direct to check this map factors through
\[ D^{2,\lambda,\chi}_r \to \bigoplus_{g \in G / T \cdot G^{p^n}} W^{n, g}_{\lambda, \chi}, \]
which is an inverse of the previously defined map $\bigoplus_{g \in G / T \cdot G^{p^n}} W^{n, g}_{\lambda, \chi} \to D^{2,\lambda,\chi}_r$.
\end{proof}

For $r = \sqrt[p^n]{1 / p}$ and each $g \in G / T \cdot G^{p^n}$, we define a projection $p^n_g : D^{2,\lambda,\chi}_r \to W^{n, g}_{\lambda, \chi}$ by Prop \ref{Ugn decomposition}.

\begin{prop}\label{M decomposition}
If $\lambda$ is generic and $r = \sqrt[p^n]{1 / p}$ for $n \in \ZZ_{\geq 1}$, any submodule $M \subset D^{2,\lambda,\chi}_r$ decomposes as
\[ M \cong \bigoplus_{g \in G / T \cdot G^{p^n}} p^n_g(M). \]
If $n' < n$ with $r' := \sqrt[p^{n'}]{1 / p}$, $M' := M \otimes_{D^2_r} D^2_{r'}$, we define the map $\mathrm{pr}^n_{n'}$ to make the following diagram commutate
\[ \xymatrix@C=3cm{
M \ar[r]^{\cong} \ar[d] & \bigoplus_{g \in G / T \cdot G^{p^n}} p^n_g(M) \ar[d]^{\mathrm{pr}^n_{n'}}\\
M' \ar[r]^{\cong} & \bigoplus_{g' \in G / T \cdot G^{p^{n'}}} p^{n'}_{g'}(M').\\
}\]
It maps $p^n_g(M)$ to $p^{n'}_{\bar{g}}(M')$, equivalently, \[ \mathrm{pr}^n_{n'}(p^n_g(M)) \subset p^{n'}_{\bar{g}}(M'), \] where $g \in G / T \cdot G^{p^n}$ lifts $\bar{g} \in G / T \cdot G^{p^{n'}}$.
\end{prop}
\begin{proof}
For the first part, we discuss three possible cases.
\begin{itemize}
  \item When $W^{n, \mathrm{id}}_{\lambda, \chi}$ is simple. The set of $\wUgnK$-modules $\{W^{n, g}_{\lambda, \chi} | g \in G / T \cdot G^{p^n}\}$ are all simple as they are conjugated to each other by $G$ actions. They are also pairwise nonisomorphic to each other by Thm \ref{non isomorphic}. We can conclude that $D^{2,\lambda,\chi}_r$ is simple as a $D^{2,\chi}_r$-module, and $M$ is either zero or $D^{2,\lambda,\chi}_r$ itself. The claim is clear from Prop \ref{Ugn decomposition}.
  \item When $W^{n, \mathrm{id}}_{\lambda, \chi}$ is not simple but decomposable. By Thm \ref{length 2}, $W^{n, \mathrm{id}}_{\lambda, \chi}$ decomposes as $W^{n, \mathrm{id}}_{\lambda, \chi, 1} \oplus W^{n, \mathrm{id}}_{\lambda, \chi, 2}$. For any $g \in G / T \cdot G^{p^n}$, $g^{-1} W^{n, \mathrm{id}}_{\lambda, \chi} g$ decomposes as $g^{-1} W^{n, \mathrm{id}}_{\lambda, \chi, 1} g \oplus g^{-1} W^{n, \mathrm{id}}_{\lambda, \chi, 2} g$. We construct 
  \[ D^{2,\lambda,\chi}_{r, 1} := \bigoplus_{g \in G / T \cdot G^{p^n}} g^{-1} W^{n, \mathrm{id}}_{\lambda, \chi, 1} g, ~ D^{2,\lambda,\chi}_{r, 2} := \bigoplus_{g \in G / T \cdot G^{p^n}} g^{-1} W^{n, \mathrm{id}}_{\lambda, \chi, 2} g \] two summands of $D^{2,\lambda,\chi}_r$ via (\ref{D2 decomposition}). Again by \ref{length 2}, $D^{2,\lambda,\chi}_r$ is of length at most 2. The two submodules $D^{2,\lambda,\chi}_{r, 1}$, $D^{2,\lambda,\chi}_{r, 2}$ have to be simple and the only two nonzero proper submodules of $D^{2,\lambda,\chi}_r$.
  \item When $W^{n, \mathrm{id}}_{\lambda, \chi}$ is not simple and indecomposable. By Thm \ref{length 2}, $W^{n, \mathrm{id}}_{\lambda, \chi}$ has a unique nonzero proper submodule $W^{n, \mathrm{id}}_{\lambda, \chi, \mathrm{sub}}$. For any $g \in G / T \cdot G^{p^n}$, $g^{-1} W^{n, \mathrm{id}}_{\lambda, \chi, \mathrm{sub}} g$ is the unique nonzero proper submodule of $W^{n, g}_{\lambda, \chi}$. We construct 
  \[ D^{2,\lambda,\chi}_{r, \mathrm{sub}} := \bigoplus_{g \in G / T \cdot G^{p^n}} g^{-1} W^{n, \mathrm{id}}_{\lambda, \chi, \mathrm{sub}} g \] as a nonzero proper submodule of $D^{2,\lambda,\chi}_r$ via (\ref{D2 decomposition}). Same as above, $D^{2,\lambda,\chi}_r$ is of length at most 2. The submodule $D^{2,\lambda,\chi}_{r, \mathrm{sub}}$ has to be its only nonzero proper submodule.
\end{itemize}
We have verified the first part of the proposition for all three cases.

For the second part, the natural embedding $M \hookrightarrow D^{2,\lambda,\chi}_r$ induces
\[ \xymatrix@C=3cm{
  \bigoplus_{g \in G / T \cdot G^{p^n}} p^n_g(M) \ar[r] \ar[d]^{\mathrm{pr}^n_{n'}} & D^{2,\lambda,\chi}_r \ar[d] \\
  \bigoplus_{g' \in G / T \cdot G^{p^{n'}}} p^{n'}_{g'}(M') \ar[r] & D^{2,\lambda,\chi}_{r'}.\\
}\]
It is reduced to check the statement for $D^{2,\lambda,\chi}_r \to D^{2,\lambda,\chi}_{r'}$, which is clear from Prop \ref{Ugn decomposition}.
\end{proof}

Back to the Banach representation $V^2_\lambda$ of $U^2_D$, we consider its locally analytic vectors
\[ V^{2,\la}_\lambda = \{ f \in C^\la(U^2_D, \QQ_p) ~ | ~ f(gt) = \lambda(t)f(g) ~~ \for ~ \forall ~ t \in T^2_D,  ~ g \in D^\times \}, \] and its subrepresentation
\begin{eqnarray*} V^{2,\la}_{\lambda, \chi} & = & \{z \cdot v = \chi(z) v | v \in V^{2,\la}_\lambda \otimes_{\QQ_p} K, ~ \forall z \in Z(\frak{g}_K) \} \\ & = & \{ f \in C^\la(U^2_D, K) ~ | ~ f(gt) = \lambda(t)f(g) ~~ \for ~ \forall ~ t \in T^2_D, ~ g \in D^\times, ~ z \in Z(\frak{g}_K) \} \end{eqnarray*} cut off by the infinitesimal character $\chi$.

These locally analytic representations $V^{2,\la}_\lambda, V^{2,\la}_{\lambda, \chi}$ are respectively dual to $D(U^2_D, K)$-modules $D^{2,\lambda}, D^{2,\lambda, \chi}$ (\cite[Cor 3.3]{ST02J}).

\begin{lem}\label{nonzero la rep}
The locally analytic representation $V^{2,\la}_{\lambda, \chi}$ is nonzero.
\end{lem}
\begin{proof}
Its dual $D^{2,\lambda, \chi}$ is nonzero as it has a nonzero vector $1$ mapping to $0 \neq 1 \in W_{\lambda,\chi}$.
\end{proof}

\begin{prop}\label{restriction invariance}
Suppose $M$ is a closed subrepresentation of $V^{2,\la}_{\lambda, \chi}$.
If $f \in M$, $n \in \ZZ_{\geq 1}$, we have $f \cdot \chi_U \in M$ for any $T \cdot G^{p^n}$-coset $U$, where $\chi_U$ is the characteristic function vanishing outside $U$.
\end{prop}
\begin{proof}
Let $M^\perp$ be the submodule of $D^{2,\lambda,\chi}$ dual to $V^{2,\la}_{\lambda, \chi} / M$. It suffices to prove $\delta(f \cdot \chi_{T \cdot G^{p^n}}) = 0$ for all $\delta \in M^\perp$ and $n \geq 1$. For any $N \geq n$, we base change $M^\perp$ to $M^\perp_N := D^2_{r_N} \otimes_{D^{2,\lambda,\chi}} M^\perp$ where $r_N := \sqrt[p^N]{1 / p}$.

We define \[ p^N_n(\delta) := \sum_{g \in T \cdot G^{p^n} / T \cdot G^{p^N}} p^N_{g}(\delta), \]  where $p^N_{g}$ is the projection defined in front of Prop \ref{M decomposition}. 
By Prop \ref{M decomposition}, $p^N_n(\delta) \in M^\perp_N$ for all $N \geq n$. We define \[ p_n(\delta) := \varprojlim \big(p^N_n(\delta)\big) \in \varprojlim_{N \geq n} M^\perp_N \simeq M^\perp. \]
We have \[ \delta(f \cdot \chi_{T \cdot G^{p^n}}) = p_n(\delta)(f) = 0 \] for $\forall \delta \in M^\perp$, therefore $f \cdot \chi_{T \cdot G^{p^n}} \in M$.
\end{proof}

\begin{theorem}\label{la dense}
Suppose $\lambda$ is generic. If $M \subset V^{2,\la}_{\lambda, \chi}$ is a nonzero closed subrepresentation for an infinitesimal character $\chi$. We claim $M$ is dense in $V^2_\lambda \otimes_{\QQ_p} K$.
\end{theorem}
\begin{proof}
Let $X := D^\times / T_D$. There is a splitting $X \hookrightarrow D^\times \twoheadrightarrow X$ as $p$-adic manifolds such that 
$V^2_\lambda \otimes_{\QQ_p} K \cong C(X, K)$ as topological $K$-vector spaces.
We pick a nonzero function $f \in M$. The pointwise product of $f$ with any smooth function still has the infinitesimal character $\chi$. We have $f \cdot C^{\text{sm}}(D^\times, K) \subset V^{2,\la}_{\lambda, \chi}$ is dense in $V^2_\lambda \otimes_{\QQ_p} K \cong C(X, K)$. By Prop \ref{restriction invariance}, $f \cdot C^{\text{sm}}(D^\times, K)$ is included in $M$.
\end{proof}

\begin{cor}\label{abs irr}
  Suppose $\lambda$ is generic. The admissible Banach representation $V^2_\lambda$ is absolutely irreducible.
\end{cor}
\begin{proof}
  Suppose $V$ is a nonzero subrepresentation of $V^2_\lambda \otimes_{\QQ_p} K_0$ for a finite extension $K_0$ over $\QQ_p$.
  By \cite[Thm 7.1]{ST03}, $M^\la$ is nonzero in $V^2_\lambda \otimes_{\QQ_p} K_0$. There exists an irreducible subrepresentation $M_0$ of $V^\la$. By \cite[Cor 3.14, Cor 3.10]{DS13}, there exists a closed absolutely irreducible subrepresentation $M$ of $V^\la \otimes_{K_0} K \subset V^{2,\la}_{\lambda, \chi}$ with an infinitesimal character $\chi$ for a finite extension $K$ over $K_0$.
  By Thm \ref{la dense}, $M$ is dense in $V^2_\lambda \otimes_{\QQ_p} K$. As $M$ is contained in $V \otimes_{\QQ_p} K$, the subrepresentation $V$ has to be $V^2_\lambda \otimes_{\QQ_p} K_0$ itself.
\end{proof}

\begin{prop}\label{no inf char}
  The locally analytic vectors $V^{2,\la}_\lambda$ of the absolutely irreducible admissible Banach representation $V^2_\lambda$ as a locally analytic representation does not admit an infinitesimal character.
\end{prop}
\begin{proof}
This is clear from Lem \ref{nonzero la rep} as $V^2_\lambda$ contains vectors of any infinitesimal character. 
\end{proof}

\begin{proof}[Proof of Theorem \ref{counterexamples}]
The representation $V_\lambda$ can be further decomposed as 
\[ V_\lambda \cong \bigoplus_{c \in \cO_D^\times / (T^1_D \times \mu_{p-1} \cdot U^2_D)} V^2_{c \circ \lambda}, ~ \mathrm{where} \]
\[ V^2_{c \circ \lambda} = \{ f \in C(U^2_D \cdot c, \QQ_p) ~ | ~ f(gt) = \lambda(t)f(g) ~~ \for ~ \forall ~ t \in T^2_D, ~ g \in U^2_D \} \]
as a representation of $U^2_D$, similar to the identification in Lem \ref{V lambda adm}.
Since $\lambda$ is generic in the sense of Def \ref{infinitesimal genericity}, all $c \circ \lambda$ are generic.
All $V^2_{c \circ \lambda}$ are absolutely irreducible admissible Banach representation of $U^2_D$ by Cor \ref{abs irr}. 
The representation $V_\lambda$ is of finite length and semisimple over $U^2_D$, so is it over $D^\times$.

Consider any absolutely irreducible admissible Banach subrepresentation $V_{\lambda, \bullet}$ of $V_\lambda$ over $D^\times$ (may go to a finite extension of $\QQ_p$ by \cite[Thm 1.1]{DS13}).
$V_{\lambda, \bullet}$ is a direct sum of some of $V^2_{c \circ \lambda}$ over $U^2_D$.
Similar to Prop \ref{no inf char}, its locally analytic vectors $V_{\lambda, \bullet}^\la$ do not have finite length nor admit a central character as $V_{\lambda, \bullet}^\la$ contains vectors of any infinitesimal character by Lem 5.5.
\end{proof}

\end{document}